\documentclass[a4paper, 12pt, reqno]{amsart}
\usepackage{amsmath,amsthm,amssymb,latexsym,mathrsfs,stmaryrd,color,mathtools}
\usepackage[all]{xy}
\usepackage{url}
\usepackage[a4paper, margin=1.2in]{geometry}
\usepackage{enumerate}
\usepackage[utf8]{inputenc}
\usepackage[T1]{fontenc}

\usepackage{relsize}
\usepackage[bbgreekl]{mathbbol}
\usepackage{amsfonts}

\DeclareSymbolFontAlphabet{\mathbb}{AMSb}
\DeclareSymbolFontAlphabet{\mathbbl}{bbold}

\usepackage{enumitem}
\usepackage[colorlinks=true, hyperindex, linkcolor=magenta, pagebackref=false, citecolor=cyan, pdfpagelabels]{hyperref}
\usepackage[capitalize]{cleveref}
\setlist[enumerate]{itemsep=2pt,parsep=2pt,before={\parskip=2pt}}

\newcommand{\cosimp}[3]{\xymatrix@1{#1 \ar@<.4ex>[r] \ar@<-.4ex>[r] & {\ }#2 \ar@<0.8ex>[r] \ar[r] \ar@<-.8ex>[r] & {\ } #3 \ar@<1.2ex>[r] \ar@<.4ex>[r] \ar@<-.4ex>[r] \ar@<-1.2ex>[r] & \cdots }}

\newcommand{\adjunction}[4]{\xymatrix@1{#1{\ } \ar@<0.3ex>[r]^{ {\scriptstyle #2}} & {\ } #3 \ar@<0.3ex>[l]^{ {\scriptstyle #4}}}}

\numberwithin{equation}{section}

\DeclareMathOperator{\crys}{crys}

\DeclareMathOperator{\Hom}{Hom}

\DeclareMathOperator{\Spf}{Spf}

\DeclareMathOperator{\ad}{ad}
\DeclareMathOperator{\cont}{cont}

\DeclareMathOperator{\GL}{GL}
\DeclareMathOperator{\GSp}{GSp}

\DeclareMathOperator{\rig}{rig}

\DeclareMathOperator{\reg}{reg}

\DeclareMathOperator{\Aut}{Aut}

\DeclareMathOperator{\rec}{rec}
\DeclareMathOperator{\Gal}{Gal}

\DeclareMathOperator{\der}{der}
\DeclareMathOperator{\ab}{ab}

\newtheorem{theorem}{Theorem}[section]
\newtheorem*{theorem*}{Theorem}
\newtheorem*{definition*}{Definition}

\newtheorem{lemma}[theorem]{Lemma}

\theoremstyle{definition}
\newtheorem{definition}[theorem]{Definition}

\newtheorem{remark}[theorem]{Remark}

\crefname{assumption}{assumption}{assumptions}
\crefname{construction}{construction}{constructions}

\title[Crystalline lifts with fixed determinant]{Crystalline lifts of semisimple $G$-valued Galois representations with fixed determinant}
\author{Kensuke Aoki}
\address{Department of Mathematics, Faculty of Science, Kyoto University}
\email{aoki.kensuke.88s@st.kyoto-u.ac.jp}

\begin{document}

\begin{abstract} 
%We show the existence of a crystalline lift with a fixed determinant representation of split reductive $G$-valued representations $\overline{\rho} \colon \Gal_K \to G(\mathbb{F}_p)$ of the Galois group $\Gal_K$ of a finite extension $K/\mathbb{Q}_p$. 
For a finite extension $K/\mathbb{Q}_p$ and a split reductive group $G$ over $\mathcal{O}_K$, let $\overline{\rho} \colon \Gal_K \to G(\overline{\mathbb{F}}_p)$ be a continuous quasi-semisimple mod $p$ $G$-valued representation of the absolute Galois group $\Gal_K$. 
Let $\overline{\rho}^{\ab}$ be the abelianization of $\overline{\rho}$ and fix a crystalline lift $\psi$ of $\overline{\rho}^{\ab}$. 
We show the existence of a crystalline lift $\rho$ of $\overline{\rho}$ with regular Hodge-Tate weights such that the abelianization of $\rho$ coincides with $\psi$. 
We also show analogous results in the case that $G$ is a quasi-split tame group and $\overline{\rho} \colon \Gal_K \to {^L}G(\overline{\mathbb{F}}_p)$ is a semisimple mod $p$ $L$-parameter. 
These theorems are generalizations of those of Lin and B\"ockle-Iyengar-Pa\v{s}k\={u}nas. 
\end{abstract}

\maketitle

\tableofcontents

\section{Introduction}

Let $K/\mathbb{Q}_p$ be a finite extension with its ring of integers $\mathcal{O}_K$, $\Gal_K \coloneqq \Gal(\overline{K}/K)$ be the absolute Galois group of $K$, $L/K$ be a finite extension with its ring of integers $\mathcal{O}_L$. Let $\mathfrak{m}_K$, $\mathfrak{m}_L$ be the maximal ideals of $\mathcal{O}_K$, $\mathcal{O}_L$, and $k$, $k_L$ be the residue fields of $\mathcal{O}_K$, $\mathcal{O}_L$, respectively. 
Let $\mathbb{C}_p$ be the $p$-adic completion of $\overline{K}$. 

%Recall that a continuous $G$-valued representation $\rho \colon \Gal_K \to G(L)$ is said to be \emph{crystalline} if for some (equivalently, any) finite dimensional faithful algebraic $L$-representation $i \colon G \otimes_{\mathcal{O}_K} L \to \GL(V)$, the composite $i \circ \rho$ is crystalline (cf.\ \cite[Definition 2.3.1]{Bal12}). 

For a reductive group $G$ over $\mathcal{O}_L$, a continuous $G$-valued representation $\rho \colon \Gal_K \to G(L)$ is \emph{Hodge-Tate} (resp.\ \emph{crystalline}, \emph{potentially crystalline}) if for any algebraic $L$-representation $\tau \colon G \to \GL (V)$, 
%and $\sigma \in \Sigma_L \coloneqq \Hom_{\mathbb{Q}_p} (L, \mathbb{C}_p)$, 
the composite $\tau \circ \rho$ is a Hodge-Tate (resp.\ crystalline, potentially crystalline) representation.  
For a Hodge-Tate $G$-valued representation $\rho \colon \Gal_K \to G(L)$, we say that $\rho$ has \emph{regular} Hodge-Tate weights if for each $\sigma \in \Hom_{\mathbb{Q}_p} (L, \mathbb{C}_p)$, the associated cocharacter $\mathcal{H}\mathcal{T}(\rho)^{\sigma}$ of $G_{\mathbb{C}_p}$ is regular, that is, not killed by any root of $G_{\mathbb{C}_p}$ (see Definition \ref{G_Hodge_Tate}). 
Let $G^{\ab}$ be the abelianization of $G$. The composite of $\rho$ with $G(L) \to G^{\ab} (L)$ is written as $\rho^{\ab}$. We call it the \emph{abelianization} of $\rho$. 

A continuous representation $\rho \colon \Gal_K \to G(\mathcal{O}_L)$ is said to be \emph{Hodge-Tate} (resp.\ \emph{crystalline}, \emph{potentially crystalline}) if the composite $\Gal_K \xrightarrow{\rho} G(\mathcal{O}_L) \hookrightarrow G(L)$ is Hodge-Tate (resp.\ crystalline, potentially crystalline). 

From here, let $G$ be a connected split reductive group over $\mathcal{O}_L$ and let $T$ be a maximal torus of $G$. 
We say $\overline{\rho} \colon \Gal_K \to G(\overline{\mathbb{F}}_p)$ is \emph{quasi-semisimple} if there exists a maximal torus $T'$ of $G_{\overline{\mathbb{F}}_p} \coloneqq G \otimes_{\mathcal{O}_K} \overline{\mathbb{F}}_p$ such that $\overline{\rho} (I_K) \subset T'(\overline{\mathbb{F}}_p)$ and $\overline{\rho} (\Gal_K) \subset N_{G_{\overline{\mathbb{F}}_p}} (T')(\overline{\mathbb{F}}_p)$, where $I_K \subset \Gal_K$ denotes the inertia subgroup and $N_{G_{\overline{\mathbb{F}}_p}} (T')$ denotes the normalizer of $T'$ in $G_{\overline{\mathbb{F}}_p}$. 

The first main theorem of this paper is as follows: 

\begin{theorem}[{Theorem \ref{qss_cryslift_fixeddet}}]
\label{intro_qss_cryslift_fixeddet}

%Assume that $T$ is a split torus and the composite $T \hookrightarrow G \twoheadrightarrow G^{\ab}$ has a right inverse, that is, there exists a homomorphism $G^{\ab} \to T$ such that the composite $G^{\ab} \to T \to G^{\ab}$ is the identity.  
Let $\overline{\rho} \colon \Gal_K \to G(k_L)$ be a quasi-semisimple $G$-valued representation such that $\overline{\rho} (I_K) \subset T(k_L)$ and $\overline{\rho} (\Gal_K) \subset N_G (T)(k_L)$.
Let $\overline{\rho}^{\ab} \colon \Gal_K \to G^{\ab} (k_L)$ be the composite of $\overline{\rho}$ and the natural map $G(k_L) \to G^{\ab} (k_L)$. Let $\psi \colon \Gal_K \to G^{\ab} (\mathcal{O}_L)$ be a crystalline lift of $\overline{\rho}^{\ab}$. Then there exist a finite extension $L'/L$ and a crystalline lift $\rho \colon \Gal_K \to G(\mathcal{O}_{L'})$ of $\overline{\rho}$ with regular Hodge-Tate weights such that ${\rho}^{\ab} = \psi$. 
%of the ring of integers $\mathcal{O}_{L'}$ of $L'$. 
%There also exists a crystalline lift ${\rho}^{\reg}$ with regular Hodge-Tate weights such that ${\rho}^{\reg, \det} = \chi^{q^M} \psi^f$ for a crystalline representation $\chi \colon \Gal_K \to G^{\ab} (\mathcal{O}_L)$ and a sufficiently large positive integer $M$. 
Moreover, if we do not impose a crystalline lift $\rho$ has regular Hodge-Tate weights, $L'/L$ can be taken to be a finite unramified extension. 
\end{theorem}

Theorem \ref{intro_qss_cryslift_fixeddet} is a generalization of the results of Lin \cite[Theorem 5, Theorem 6]{Lin20} and B\"ockle-Iyengar-Pa\v{s}k\={u}nas \cite[Proposition 2.6]{BIP22}. 

We also prove analogous results of Theorem \ref{intro_qss_cryslift_fixeddet} for potentially crystalline lifts of semisimple mod $p$ $L$-parameters for quasi-split tame groups over $K$. 
Let $G$ be a connected quasi-split reductive group over $K$. We assume that $G$ is tame, that is, $G$ splits over a tame Galois extension $E$ of $K$. Write $\widehat{G}$ for the Langlands dual group of $G$ and ${^L}G = \widehat{G} \rtimes \Gal_K$ for the $L$-group of $G$. 
Let $\widehat{T}$ be a maximal torus of $\widehat{G}$. 

The second main theorem of this paper is as follows:

\begin{theorem}[{Theorem \ref{qst_dRlift_fixeddet}}]
\label{intro_qst_dRlift_fixeddet}
%Let $K/\mathbb{Q}_p$ be a finite extension and 
Let $\overline{\rho} \colon \Gal_K \to {^L}G (\overline{\mathbb{F}}_p)$ be a semisimple mod $p$ $L$-parameter in the sense of \cite[Definition 3.2.5]{Lin23b}. 
Fix a maximal torus $S$ of $G$ such that $\overline{\rho}$ factors through ${^L}S (\overline{\mathbb{F}}_p)$ (see Theorem \ref{LS_factorization}).
%, which exists by Theorem \ref{LS_factorization}. 
%Assume that the composite $S \hookrightarrow G \twoheadrightarrow G^{\ab}$ has a right inverse, which is written as $i \colon G^{\ab} \to S$. 
Let $i \colon G^{\ab} \to S$ be a right inverse of the composite $S \hookrightarrow G \twoheadrightarrow G^{\ab}$. 
%Write $\overline{\psi} = \overline{\rho}^{\ab}$ for the composite of $\overline{\rho}$ and the map ${^L}G (\overline{\mathbb{F}}_p) \to {^L} G^{\ab}(\overline{\mathbb{F}}_p)$ induced by the composite $G^{\ab} \xrightarrow{i} S \hookrightarrow G$ and let $\psi \colon \Gal_K \to {^L} G^{\ab}(\overline{\mathbb{Z}}_p)$ be a potentially crystalline lift of $\overline{\psi}$.
For $A = \overline{\mathbb{F}}_p^{\ast}$ or $\overline{\mathbb{Z}}_p^{\ast}$ and an $L$-parameter $\tau \colon \Gal_K \to {^L}S(A)$, write $\tau^{\ab}$ for the composite of $\tau$ and the map ${^L}S (A) \to {^L}G^{\ab}(A)$ induced by the composite $G^{\ab} \xrightarrow{i} S \hookrightarrow G$ and we set $\overline{\psi} \coloneqq \overline{\rho}^{\ab}$. Let $\psi \colon \Gal_K \to {^L}G^{\ab}(\overline{\mathbb{Z}}_p)$ be a potentially crystalline lift of $\overline{\psi}$. 
Then $\overline{\rho}$ admits a potentially crystalline lift $\rho \colon \Gal_K \to {^L}G(\overline{\mathbb{Z}}_p)$ which has regular Hodge-Tate weights and satisfies $\rho^{\ab} = \psi$. 
\end{theorem}

%\begin{remark}
%If we consider the case that $\overline{\rho}$ is defined over $\overline{\mathbb{F}}_p$, Theorem \ref{intro_qst_dRlift_fixeddet} is stronger than Theorem \ref{intro_qss_cryslift_fixeddet}. But Theorem \ref{intro_qss_cryslift_fixeddet} will be useful in considering the crystalline lifting problem with a fixed determinant in a given coefficient ring $\mathcal{O}_L$. 
%\end{remark}

The existence of crystalline lifts with fixed determinant in the case of $G = \GL_m$ is proved by B\"ockle-Iyengar-Pa\v{s}k\={u}nas \cite[Proposition 2.6, Proposition 2.7]{BIP22}. In the paper, they firstly proved the case where $\overline{\rho} \colon \Gal_K \to \GL_m (\overline{\mathbb{F}}_p)$ is absolutely irreducible and the general case is deduced from the former case using the geometry of the Emerton-Gee stacks. 
Also, Lin \cite[Theorem 5, Theorem 6]{Lin20} showed the existence of crystalline lifts with regular Hodge-Tate weights of quasi-semisimple mod $p$ $G$-valued representations when $G$ is a split reductive group over $\mathcal{O}_K$ (see \cite[Theorem 5.3.1]{Lin23b} for semisimple mod $p$ $L$-parameters for quasi-split tame groups).
In this paper, we combine the methods of \cite{BIP22} and \cite{Lin20}, \cite{Lin23b} to prove the existence of crystalline lifts of $\overline{\rho}$ with fixed abelianization. 

\begin{remark}
Let $\overline{\rho} \colon \Gal_K \to G(k_L)$ be a continuous $G$-valued representation and $R_{\overline{\rho}}^{\Box}$ be the universal framed deformation ring
constructed by Balaji \cite[Theorem 1.2.2]{Bal12}. Write $\mathcal{X}_{\overline{\rho}}^{\Box} \coloneqq \Spf (R_{\overline{\rho}}^{\Box})^{\rig}$ for the rigid generic fiber of $\Spf (R_{\overline{\rho}}^{\Box})$
as constructed in \cite[Definition 7.1.3]{dJ95}.
For $G = \GSp_{2n}$, Zariski density of crystalline points on an irreducible component of $\mathcal{X}_{\overline{\rho}}^{\Box}$ which has a crystalline point is shown in \cite{Aok24a} using the geometry of the trianguline deformation spaces for $\GSp_{2n}$. 
To prove the full version of the Zariski density theorem for $\GSp_{2n}$, we have to show that each irreducible component of $\mathcal{X}_{\overline{\rho}}^{\Box}$ has a crystalline point. 
By the results of Pa\v{s}k\={u}nas-Quast \cite{PQ24b}, it suffices to show that any mod $p$ $G$-valued representation $\overline{\rho}$ has a crystalline lift with fixed abelianization. We expect that Theorem \ref{intro_qss_cryslift_fixeddet} is the first step for the full version of the Zariski density theorem for split reductive groups. 
\end{remark}

The strategies of the proofs of Theorem \ref{intro_qss_cryslift_fixeddet} and Theorem \ref{intro_qst_dRlift_fixeddet} are based on Lin's strategies in \cite{Lin20} and \cite{Lin23b}, respectively, where the existence of crystalline lifts of semisimple mod $p$ $G$-valued representations is shown. 
In order to control the abelianizations of crystalline lifts, we use the calculations of abelianizations of Galois representations via the reciprocity map of local class field theory as in the work of B\"ockle-Iyengar-Pa\v{s}k\={u}nas \cite{BIP22}. 

\subsection*{Acknowledgement}
The author thanks his advisor Tetsushi Ito for valuable comments on the draft of the paper. 
He also thanks Gebhard B\"ockle, Noriyuki Abe, and Teruhisa Koshikawa for comments on the abelianizations of reductive groups. 

\section{Notation}
Let $p$ be a prime number. 
%Mat, \mathbb{A}, (p-adic norm |-|_p)
For a group $G$ and a subset $S \subset G$, we write $N_G (S)$ for the normalizer of $S$ in $G$ and $C_G (S)$ for the centralizer of $S$ in $G$. 
We fix an algebraic closure $\overline{\mathbb{Q}}_p$ of $\mathbb{Q}_p$. Let $K$ be a finite extension of ${\mathbb{Q}}_p$ contained in $\overline{\mathbb{Q}}_p$.  
%and $K_0$ be the maximal unramified subextension of ${\mathbb{Q}}_p$ in $K$. 
%We write $\zeta_m$ for a fixed $m$-th roots of unity.
%Let $K_n \coloneqq K(\zeta_{p^n})$ and $K_{\infty} \coloneqq \bigcup_n K_n$. Let $k$ be the residue field of $K$, $k'$ be the residue field $K_{\infty}$ and $K'_0 \coloneqq W(k')$. We write $\Gamma \coloneqq \Gal (K_{\infty} / K)$. 
%Let $|-|_K$ be the $p$-adic norm on $K$ normalized by $|p|_K = p^{-1}$. 
Let $\mathbb{C}_p$ be the $p$-adic completion of $\overline{\mathbb{Q}}_p$. Let $L$ be a sufficiently large finite extension of $K$ such that |$\Hom_{\mathbb{Q}_p} (K, L)| = [K \colon \mathbb{Q}_p]$. Let $k_L$ be the residue field of $L$ and $\Sigma_L \coloneqq \Hom_{{\mathbb{Q}}_p} (L, \mathbb{C}_p)$. 
Let ${\Gal}_K \coloneqq \Gal(\overline{\mathbb{Q}}_p/K)$ be the absolute Galois group of $K$, and $I_K$ be the inertia subgroup of $\Gal_K$. 
%Let $\epsilon \colon {\Gal}_K \to {\mathbb{Z}}_p^{\ast}$ be the $p$-adic cyclotomic character. 
Let $W_K$ be the Weil group of $K$. For a finite Galois extension $E/K$, let $W_{E/K} \coloneqq W_K / \overline{W_E^{\der}}$, where $\overline{W_E^{\der}}$ is the closure of the derived subgroup $W_E^{\der}$ of $W_E$, be the relative Weil group as defined in \cite{Tat79}. Write $\widehat{\mathbb{Z}} \coloneqq \prod_{\ell \colon \mathrm{prime}} \mathbb{Z}_{\ell}$. 
We fix a uniformizer $\varpi_K \in {\mathcal{O}}_K$ and normalize the reciprocity isomorphism $\rec_K \colon K^{\ast} \xrightarrow{\sim} W_K^{\text{ab}}$ of local class field theory such that $\varpi_K$ is mapped to the geometric Frobenius automorphism. 
%For a $(\varphi, \Gamma)$-module $D$, we define a Herr complex $C_{\Herr}^{\bullet} (D)$ as 
%\[
%D \xrightarrow{(\varphi - 1, \gamma - 1)} D \oplus D \xrightarrow{(1 - \gamma) \oplus (\varphi - 1)} D
%\]
%where $\gamma$ is a topological generator of $\Gamma$ and write $H_{\Herr}^i (D)$ for $i$-th cohomology of the complex $C_{\Herr}^{\bullet} (D)$ for $i \in [0, 2]$. 

For an algebraic group $G$ over a commutative ring $R$, write $G^{\der} \coloneqq [G, G]$ for the derived subgroup of $G$, $G^{\ab} \coloneqq G / G^{\der}$ for the abelianization of $G$, and $G^{\circ}$ for the identity component of $G$. 
Write $X^{\ast} (G) \coloneqq \Hom_R (G, \mathbb{G}_{m, R})$ for the character group of $G$ and $X_{\ast} (G) \coloneqq \Hom_R (\mathbb{G}_{m, R}, G)$ for the cocharacter group of $G$. 

%Let $\det \colon G \twoheadrightarrow G^{ab}$ be the abelianization map. 
Let $A$ be a commutative topological $R$-algebra. 
%For a continuous representation $\rho \colon \Gal_K \to G(A)$ (resp.\ a continuous $L$-parameter $\rho \colon \Gal_K \to {^L}G(A))$, let $\rho^{\ab} \colon \Gal_K \to G^{\ab}(A)$ (resp.\ $\rho^{\ab} \colon \Gal_K \to {^L}G^{\ab}(A)$) be the composite of $\rho$ with the abelianization map $G(A) \to G^{\ab}(A)$ (resp.\ ${^L}G(A) \to {^L}G^{\ab}(A)$).  
For a continuous representation $\rho \colon \Gal_K \to G(A)$, let $\rho^{\ab} \colon \Gal_K \to G^{\ab}(A)$ be the composite of $\rho$ with the abelianization map $G(A) \to G^{\ab}(A)$. 
%We call $\overline{\rho}^{\det}$ the \emph{determinant} of $\overline{\rho}$. 

%Witt ring

%Let $\GSp_{2n}$ be the split reductive group over $\mathbb{Z}$ defined by 
%\[
%\GSp_{2n} (R) = \{ g \in \GL_{2n} (R) \ | \ {^t}gJg = \text{sim} (g) J \ \text{for some} \ \text{sim} (g) \in R^{\ast} \}
%\]
%for any commutative ring $R$, where $J$ is defined by $\mathbf{x} = (x_{ij})_{i, j}$ in $\GL_{2n}$ such that $x_{ij} = (-1)^{\text{sgn} (j - i)}$ if $i + j = 2n + 1$ and otherwise $x_{ij} = 0$. There is the natural faithful algebraic representation $r_{\std} \colon \GSp_{2n} \to \GL_{2n}$ of the group scheme over $\mathbb{Z}$ by the definition. 
%We define the standard Borel subgroup as $B = \GSp_{2n} \cap B_{2n}$ and the maximal split torus $T = \GSp_{2n} \cap T_{2n}$ of $\GSp_{2n}$, embedding $\GSp_{2n}$ into $\GL_{2n}$ by $r_{\std}$. 
%We fix the isomorphism $T \cong {\mathbb{G}}_m^{n+1}$ such that $(t_1, \ldots, t_{n+1}) \in {\mathbb{G}}_m^{n+1} \cong T$ is sent to the element 
%\[
%(t_1, \ldots, t_n, t_n^{-1} t_{n+1}, \ldots, t_1^{-1} t_{n+1}) \in T_{2n}
%\]
%by the embedding $T \to T_{2n}$. 

\section{Crystalline extensions to normalizers of tori}

In this section, let $G$ be a connected split reductive group over $\mathcal{O}_K$ and $T$ be a maximal torus of $G$. 
%We assume that the composite $T \hookrightarrow G \to G^{\ab}$ has a right inverse, that is, there exists a homomorphism $G^{\ab} \to T$ such that the composite $G^{\ab} \to T \to G^{\ab}$ is the identity. 
%Note that $G^{\ab} \cong Z^{\circ} / (Z^{\circ} \cap G^{\der})$, where $Z$ is the center of $G$, is a torus. 

%\subsection{Extendability of torus-valued Galois representations}
We recall the notation of \cite[\S 4]{Lin20}. 
%Weylの性質あたりから書いておく
%Fix an algebraic closure $\overline{K}$ of $K$. 
Let $\Phi_K \in \Gal_K$ be a lift of a topological generator of $\Gal_K / I_K \cong \widehat{\mathbb{Z}}$. Let $W(G, T)$ be the Weyl group scheme of $(G, T)$. For the ease of notation, write $W(G, T)$ for the group $W(G, T)(\mathcal{O}_L) \cong W(G, T)(k_L)$, which is a subgroup of the finite group $W(G, T)(\overline{L})$. 

Let $M_{T, \crys}$ denote the set of continuous $T(\mathcal{O}_L)$-valued representations $I_K \to T(\mathcal{O}_L)$ of the inertia group $I_K$, which can be extended to a crystalline representation ${\Gal}_{K'} \to T(\mathcal{O}_L)$ for some finite unramified extension $K' / K$ inside $\overline{K}$. 
Let $M_{T, k_L}$ denote the set of continuous mod $\varpi_L$ representations $I_K \to T(k_L)$ of $I_K$.

Both $M_{T, \crys}$ and $M_{T, k_L}$ are naturally equipped with the structures of $\mathbb{Z}[W(G, T)]$-modules defined by $wv \coloneqq (\sigma \mapsto wv(\sigma)w^{-1})$ for $w \in W(G, T)$ and $v \in M_{T, \crys}$ (or $v \in M_{T, k_L}$). Similarly, they are naturally equipped with the structures of $\mathbb{Z}[\Gal_K / I_K]$-modules defined by $\alpha v \coloneqq (\sigma \mapsto v(\alpha^{-1} \sigma \alpha))$ for $\alpha \in \Gal_K$ and $v \in M_{T, \crys}$ (or $v \in M_{T, k_L}$). 

%Both $M_{T, \crys}$ and $M_{T, k_L}$ naturally equip the abelian group structures and a $\mathbb{Z}[W(G, T)]$-module (resp.\ $\mathbb{Z}[\Gal_K / I_K]$-module) structure defined by $wv \coloneqq (\sigma \mapsto wv(\sigma)w^{-1})$ (resp.\ $\alpha v \coloneqq (\sigma \mapsto v(\alpha^{-1} \sigma \alpha))$) for $w \in W(G, T)$ and $v \in M_{T, \crys}$ (resp.\ for $\alpha \in \Gal_K$ and $v \in M_{T, \crys}$). 

\begin{definition}
\begin{enumerate}
\item A continuous $G$-valued representation $\overline{\rho} \colon \Gal_K \to G(\overline{\mathbb{F}}_p)$ is said to be \emph{elliptic} if it does not factor through any $P(\overline{\mathbb{F}}_p) \subset G(\overline{\mathbb{F}}_p)$ where $P$ is a parabolic subgroup of $G_{\overline{\mathbb{F}}_p}$. 

\item A subgroup $\Gamma \subset G(\overline{\mathbb{F}}_p)$ is said to be \emph{$G$-completely reducible} if for any parabolic subgroup $P$ of $G_{\overline{\mathbb{F}}_p}$ such that $\Gamma \subset P(\overline{\mathbb{F}}_p)$, a Levi subgroup $L$ of $P$ also satisfies $\Gamma \subset L(\overline{\mathbb{F}}_p)$. We also say that a continuous $G$-valued representation $\overline{\rho} \colon \Gal_K \to G(\overline{\mathbb{F}}_p)$ is \emph{$G$-completely reducible} if the image of $\overline{\rho}$ is $G$-completely reducible. 

\item We say $\overline{\rho} \colon \Gal_K \to G(\overline{\mathbb{F}}_p)$ is \emph{quasi-semisimple} if there exists a maximal torus $T'$ of $G_{\overline{\mathbb{F}}_p}$ such that $\overline{\rho} (I_K) \subset T'(\overline{\mathbb{F}}_p)$ and $\overline{\rho} (\Gal_K) \subset N_{G_{\overline{\mathbb{F}}_p}} (T')(\overline{\mathbb{F}}_p)$. 
\end{enumerate}
\end{definition}

Note that an elliptic continuous $G$-valued representation $\Gal_K \to G(\overline{\mathbb{F}}_p)$ is $G$-completely reducible. 

%後々ellipticyによってfixed det crys liftを前節の定理の議論によって構成するため、前節でellipticyを導入しておいて、前節の定理をもう少し詳しく証明しておく。

\begin{theorem}
\label{reducible_quasi_semisimple}
If a continuous $G$-valued representation $\Gal_K \to G(\overline{\mathbb{F}}_p)$ is $G$-completely reducible, then $\overline{\rho}$ is quasi-semisimple. 
\end{theorem}

\begin{proof}
See \cite[Theorem 4]{Lin20}. 
\end{proof}

%For constructing crystalline lifts with a fixed determinant, 
For showing the main theorem, we will reduce to the case where $\overline{\rho} \colon \Gal_K \to G(\overline{\mathbb{F}}_p)$ is elliptic. By Theorem \ref{reducible_quasi_semisimple}, it suffices to consider the existence of a lift $\rho \colon \Gal_K \to N_G (T)(\mathcal{O}_L)$ of $\overline{\rho}$ satisfying $\rho (\Phi_K) = w^{-1}$ for an element $w \in N_G (T)(\mathcal{O}_L)$. 

In the following, we generalize Lin's argument in \cite[\S 4]{Lin20} to the case that $\rho (\Phi_K)$ is not of finite order and the homomorphism $\mathbb{Z} \to G(\mathcal{O}_L), \ 1 \mapsto \rho (\Phi_K)$ can be topologically extended to a homomorphism $\widehat{\mathbb{Z}} \to G(\mathcal{O}_L)$. 

\begin{definition}
For a Hausdorff topological group $G$, we say that a homomorphism $\mathbb{Z} \to G, \ 1 \mapsto g$ is \emph{extendable} if it can be extended to a (unique) continuous homomorphism $\widehat{\mathbb{Z}} \to G$. 
\end{definition}

\begin{lemma}
\label{G_extendability}
Let $\ell$ be a prime number, $m$ be a positive integer, $G$ be a locally profinite group, and $g \in G$. 
%For a prime number $l$, a positive integer $m$, a locally profinite group $G$ and an element $g \in G$, 
\begin{enumerate}
\item If $\lim_{k \to \infty} g^{\ell^k} = e$, the homomorphism $\mathbb{Z} \to G, \ 1 \mapsto g$ is extendable. 
\item The homomorphism $1 \mapsto g$ is extendable if and only if the homomorphism $1 \mapsto g^m$ is extendable. 
\end{enumerate}
\end{lemma}

\begin{proof}
(1) The homomorphism $\mathbb{Z} \to G, \ 1 \mapsto g$ can be extended to $\mathbb{Z}_{\ell} \to G$ because there exists a system $\{ U_i \}_{i \geq 1}$ of neighborhoods of the unity $e \in G$, where each $U_i$ is a profinite open subgroup of $G$. Then (1) follows by taking the composite of the projection $\widehat{\mathbb{Z}} \twoheadrightarrow \mathbb{Z}_{\ell}$ with a map $\mathbb{Z}_{\ell} \to G$. 

(2) The ``only if'' part is obvious. If $1 \mapsto g^m$ is extendable, there exists a continuous homomorphism $\widehat{f}' \colon m \widehat{\mathbb{Z}} \to G$ which extends $m \mathbb{Z} \to G$, that is the restriction of $f \colon \mathbb{Z} \to G,  \ 1 \mapsto g$. For each $a \in [0, m-1]$, the map 
\[
\widehat{f}'_a \colon a + m \widehat{\mathbb{Z}} \to G, \ a + mb \mapsto f(a)\widehat{f}'(mb)
\]
is continuous. Since $\{ a + m\widehat{\mathbb{Z}} \ | \ a \in [0, m-1]\}$ is an open covering of $\widehat{\mathbb{Z}}$, the map $\bigsqcup_{a \in [0, m-1]} \widehat{f}'_a \colon \widehat{\mathbb{Z}} \to G$ is a continuous homomorphism which extends $f$. 
\end{proof}

%topological torsion caseのlemma

We prove the following lemma, which is proved by Lin \cite[Lemma 4]{Lin20} when $w$ has finite order. 

\begin{lemma}
\label{extendability_NGT}
Let $\zeta \colon N_G (T) \to W(G, T)$ be the quotient map. 
%Let $$
\begin{enumerate}
\item Let $v \colon I_K \to T(\mathcal{O}_L)$ be a continuous representation such that $v \in M_{T, \crys}$. For an element $w \in N_G (T)(\mathcal{O}_L)$, the representation $v$ extends to a continuous representation $\rho \colon \Gal_K \to N_G (T)(\mathcal{O}_L)$ satisfying $\rho (\Phi_K) = w^{-1}$ and $\rho |_{I_K} = v$ if and only if 
\[
v \in \ker (M_{T, \crys} \xrightarrow{\zeta(w) \otimes 1 - 1 \otimes \Phi_K} M_{T, \crys}).
\]

\item Let $\overline{v} \colon I_K \to T(k_L)$ be a continuous representation such that $\overline{v} \in M_{T, k_L}$. 
For an element $\overline{w} \in N_G (T)(k_L)$, the representation $\overline{v}$ extends to a continuous representation $\overline{\rho} \colon \Gal_K \to N_G (T)(k_L)$ satisfying $\overline{\rho} (\Phi_K) = \overline{w}^{-1}$ and $\overline{\rho} |_{I_K} = \overline{v}$ if and only if 
\[
\overline{v} \in \ker (M_{T, k_L} \xrightarrow{\zeta(\overline{w}) \otimes 1 - 1 \otimes \Phi_K} M_{T, k_L}).
\]
\end{enumerate}
\end{lemma}

\begin{proof}
(1) Assume that $v \in \ker (\zeta(w) \otimes 1 - 1 \otimes \Phi_K \colon M_{T, \crys} \to M_{T, \crys})$. By \cite[Lemma 4]{Lin20}, the representation $v$ extends to a representation $v' \colon W_K \cong I_K \rtimes \mathbb{Z} \to G(\mathcal{O}_L)$. So it suffices to show that $v'$ extends to a continuous representation of the absolute Galois group $\Gal_K$, which is homeomorphic to the profinite set $I_K \times \widehat{\mathbb{Z}}$. 

Since the centralizer $C_G (T)$ is the maximal torus $T$ and the Weyl group $W(G, T)$ is finite, there exists a positive integer $m$ such that $w^m \in T(\mathcal{O}_L)$. For a finite extension $L'/L$, the torus $T$ splits over $\mathcal{O}_{L'}$ and the topology of the group $T(\mathcal{O}_L)$ is the restriction of that of $T(\mathcal{O}_{L'})$. Fix an isomorphism $i \colon T \cong \mathbb{G}_{m, \mathcal{O}_L}^d$ over $\mathcal{O}_{L'}$. Then there is also a positive integer $m'$ such that the $j$-th component of $i(w^{mm'}) \in (\mathcal{O}_{L'}^{\ast})^d$ is in $1 + \mathfrak{m}_{L'}$ for every $1 \leq j \leq d$. Then $\lim_{k \to \infty} (w^{mm'})^{p^k} = 1$ in $T(\mathcal{O}_{L'})$. The limit also exists and equals to $1$ since $T(\mathcal{O}_L)$ is a closed subset of $T(\mathcal{O}_{L'})$. So $1 \mapsto w^{mm'}$ and then $1 \mapsto w$ is extendable by Lemma \ref{G_extendability}. 
The ``only if'' part is obvious. 

(2) See \cite[Lemma 4 (2)]{Lin20}
%The proof of (2) is omitted because it is similar to that of (1). 
\end{proof}

We recall the notion of co-labeled Hodge-Tate characters following \cite[\S 6, Definition 3]{Lin20}. 

\begin{definition}
\label{G_Hodge_Tate}
%Fix an embedding $L \hookrightarrow \mathbb{C}_p$. 
Let $G$ be a reductive group over $L$. 
\begin{enumerate}
%\item A continuous $G$-valued representation $\rho \colon \Gal_K \to G(L)$ is \emph{Hodge-Tate} (resp.\ \emph{crystalline}, \emph{potentially crystalline}) if for any algebraic representation $\tau \colon G \to \GL (V)$, 
%and $\sigma \in \Sigma_L \coloneqq \Hom_{\mathbb{Q}_p} (L, \mathbb{C}_p)$, 
%the composite $\tau \circ \rho$ is a Hodge-Tate (resp.\ crystalline, potentially crystalline) representation.  
\item For a Hodge-Tate $G$-valued representation $\rho \colon \Gal_K \to G(L)$, there is a cocharacter $\mathcal{H}\mathcal{T}(\rho)^{\sigma} \colon \mathbb{G}_{m, \mathbb{C}_p} \to G_{\mathbb{C}_p}$ associated with $\rho$ for each $\sigma \in \Sigma_L = \Hom_{\mathbb{Q}_p} (L, \mathbb{C}_p)$ by Tannakian theory (cf.\ \cite[Definition 3]{Lin20}). 
The \emph{co-labeled Hodge-Tate character} associated with $\rho$ is defined by
\[
\mathcal{H}\mathcal{T}(\rho) \coloneqq (\mathcal{H}\mathcal{T}(\rho)^{\sigma})_{\sigma \in \Sigma_L} \in \prod_{\sigma \in \Sigma_L} X_{\ast} (G_{\mathbb{C}_p}).
\]
%We call $\mathcal{H}\mathcal{T}(\rho)$ the \emph{colabelled Hodge-Tate cocharacter of $\rho$}. 
\item A cocharacter of $G_{\mathbb{C}_p}$ is said to be \emph{regular} if it is not killed by any roots of $G_{\mathbb{C}_p}$ (with respect to a fixed maximal torus of $G_{\mathbb{C}_p}$). 
For a Hodge-Tate $G$-valued representation $\rho \colon \Gal_K \to G(L)$, we say that $\rho$ has \emph{regular} Hodge-Tate weights if for each $\sigma \in \Sigma_L$, the cocharacter $\mathcal{H}\mathcal{T}(\rho)^{\sigma}$ of $G_{\mathbb{C}_p}$ is regular. 
\end{enumerate}
\end{definition}

\section{Crystalline lifts of quasi-semisimple representations with fixed abelianizations}

We prove the first main theorem in this section. In addition to Lin's argument \cite[Proposition 2, Theorem 6]{Lin20}, we use the calculation of the norms of Galois characters due to B\"ockle-Iyengar-Pa\v{s}k\={u}nas \cite[Lemma 2.1]{BIP22} in order to construct a crystalline lift with a fixed abelianization. 

We call the homomorphism $K^{\ast} \to \mathcal{O}_K^{\ast}$ defined by $\varpi_K \mapsto 1, \ \mathcal{O}_K^{\ast} \ni a \mapsto a$ the \emph{fundamental character (of $K$)}, and the composite of the fundamental character with the reduction map $\mathcal{O}_K^{\ast} \twoheadrightarrow k^{\ast}$ the \emph{mod $\varpi_K$ fundamental character (of $K$)}. 
Write $\omega_K \colon K^{\ast} \to \mathcal{O}_K^{\ast}$ for the fundamental character of $K$. 
Note that the fundamental character is the continuous character of $K^{\ast}$ associated with the Lubin-Tate character, which is a crystalline character $\Gal_K \to \mathcal{O}_K^{\ast}$. 
Let $\mathrm{unr}(c) \colon K^{\ast} \to \mathcal{O}_L^{\ast}$ be the unramified character which sends the uniformizer $\varpi_K$ to an element $c \in \mathcal{O}_L^{\ast}$. 
We recall basic properties of crystalline characters of $\Gal_K$. 

\begin{lemma}
\label{crystalline_character}
The continuous character $K^{\ast} \to \mathcal{O}_L^{\ast}$ associated with a crystalline character $\Gal_K \to \mathcal{O}_L^{\ast}$ can be written as the twist of $\mathrm{unr}(c)$ for an element $c \in \mathcal{O}_L^{\ast}$ with 
%a product of $\Sigma_L \coloneqq \Hom_{\mathbb{Q}_p} (K, L)$-twists of the fundamental characters 
a homomorphism $x \mapsto \prod_{\sigma \colon K \to L} \sigma (\omega_K (x))^{a_{\sigma}}$ where $a_{\sigma} \in \mathbb{Z}$ for each $\sigma \in \Hom_{\mathbb{Q}_p} (K, L)$. 
\end{lemma}

\begin{proof}
See \cite[Proposition B.4]{Con11}. 
\end{proof}

\begin{lemma}
\label{fundamental_crys_lift}
Let $f := [k_L:k]$ and $K_f$ be an unramified extension of $K$ of degree $f$. Let $G$ be a connected split reductive group over $K$ and $T \subset G$ be a $K$-split maximal torus. 
Let $M_{T, \crys}^0 \subset M_{T, \crys}$ be the $\mathbb{Z}[W(G, T)] \otimes \mathbb{Z}[\Gal_K / I_K]$-submodule which consists of continuous $T(\mathcal{O}_L)$-valued representations of $I_K$ that can be extended to a continuous representation $\Gal_{K_f} \to T(\mathcal{O}_L)$. 
Then the natural reduction map $M_{T, \crys}^0 \to M_{T, k_L}$ is surjective. 
\end{lemma}

\begin{proof}
The group $M_{T, k_L}$ is generated by the mod $\varpi_{K_f}$ fundamental character. Since it lifts to the fundamental character of $K_f$ which extends to a crystalline character of $\Gal_{K_f}$ (i.e.\ the Lubin-Tate character of $\Gal_{K_f}$), the assertion follows. 
\end{proof}

The following lemma is used to take a crystalline lift which has regular Hodge-Tate weights. 

\begin{lemma}
\label{trivdet_regular}
Let $G$ be a connected split reductive group over $K$ and $T \subset G$ be a $K$-split maximal torus. Consider the map 
\[
w-1 \colon X_{\ast} (T) \to X_{\ast} (T), \ \chi \mapsto \ad(w)(\chi) \cdot \chi^{-1}
\]
for every non-trivial element $w \in W(G, T)$. Let $X_{\ast}^{\ab} \colon X_{\ast} (T) \to X_{\ast} (G^{\ab})$ be the homomorphism of the cocharacter groups induced by the composite $T \hookrightarrow G \twoheadrightarrow G^{\ab}$.  
Then the property 
\[
(R_w): \text{the rank of} \ \ker(X_{\ast}^{\ab}) / (\ker(X_{\ast}^{\ab}) \cap \ker(w-1)) \ \text{is positive}
\]
is satisfied for every non-trivial element $w \in W(G, T)$. 
\end{lemma}

\begin{proof}
Let $G^{\ad} \coloneqq G / Z^{\circ}$, which is a split reductive group. 
If $Z^{\circ}$ is trivial, the condition $(R_w)$ is satisfied for every non-trivial element $w \in W(G, T)$ of the Weyl group since $w$ acts non-trivially on $X_{\ast} (T)$. 

For general $G$, there exists a decomposition 
\[
X_{\ast} (T) \otimes_{\mathbb{Z}} \mathbb{Q} \cong (X_{\ast} (Z^{\circ}) \otimes_{\mathbb{Z}} \mathbb{Q}) \oplus (X_{\ast} (G^{\ad}) \otimes_{\mathbb{Z}} \mathbb{Q})
\]
since the composite $Z^{\circ} \hookrightarrow G \twoheadrightarrow G^{\ab}$ is an isogeny. 
Since the center of $G^{\ad}$ is trivial, the lemma follows from the case where $Z^{\circ}$ is trivial. 
%Each non-trivial element $w \in W(G, T)$ of the Weyl group 
\end{proof}

Let $(X^{\ast} (T), R, X_{\ast} (T), R^{\vee})$ be the root datum of $(G, T)$. 
The following result is well-known. 

\begin{lemma}
\label{coroot_pairing}
Let $\alpha \in X^{\ast} (T)$ be an algebraic character of $T$. Then $\alpha$ is the restriction of an algebraic character of $G$ if and only if $\langle \alpha, \beta^{\vee} \rangle = 0$ for any $\beta^{\vee} \in R^{\vee}$. 
\end{lemma}

\begin{proof}
See \cite[Part II, 1.18 (3)]{Jan07}. 
%By \cite[Proposition 8.1.8 (iii)]{Spr98}, the subtorus $T_1 \subset T$ generated by the groups $\text{Im} \beta^{\vee}$, $\beta^{\vee} \in R^{\vee}$ is a maximal torus of $G^{\der}$. 
%Note that $T \cap G^{\der} = T_1$ since $T \cap G^{\der} \supset T_1$ by \cite[Proposition 8.1.8 (iii)]{Spr98} and $T \cap G^{\der} \subset C_{G^{\der}} (T_1) = T_1$ by \cite[Corollary 7.6.4]{Spr98}. 
%So the ``only if'' part is obvious. 
%By definition, if an element $\alpha \in X^{\ast} (T)$ satisfies $\langle \alpha, \beta^{\vee} \rangle = 0$ for any $\beta^{\vee} \in R^{\vee}$ if and only if the restriction $\alpha |_{T_1}$ is trivial, which is equivalent to that $\alpha$ is the restriction of an algebraic character of $G$ since the natural map $T/(T \cap G^{\der}) \to G / G^{\der}$ is an isomorphism. 
\end{proof}

\begin{lemma}
\label{right_inverse_existence}
The composite $T \hookrightarrow G \to G^{\ab}$ has a right inverse, that is, there exists a homomorphism $G^{\ab} \to T$ such that the composite $G^{\ab} \to T \to G^{\ab}$ is the identity. 
\end{lemma}

\begin{proof}
Since the composite $T \hookrightarrow G \to G^{\ab}$ is surjective, the induced map $X^{\ast} (G^{\ab}) \to X^{\ast} (T)$ is injective. 
Let $M \coloneqq \text{im}(X^{\ast} (G^{\ab}) \to X^{\ast} (T))$. It suffices to show that the quotient $X^{\ast} (T) / M$ is torsion-free. 

The set $M$ consists of elements $\alpha \in X^{\ast} (T)$ such that $\langle \alpha, \beta^{\vee} \rangle = 0$ for any $\beta^{\vee} \in R^{\vee}$ by Lemma \ref{coroot_pairing}. 

Let $\alpha \in X^{\ast} (T)$ which satisfies $n\alpha \in M$ for an integer $n$. 
Then we have $\langle n\alpha, \beta^{\vee} \rangle = 0$ for any $\beta^{\vee} \in R^{\vee}$. 
Since $\langle n\alpha, \beta^{\vee} \rangle = n \langle \alpha, \beta^{\vee} \rangle$, we have $\langle \alpha, \beta^{\vee} \rangle = 0$ for any $\beta^{\vee} \in R^{\vee}$. 
Hence $\alpha \in M$. This shows that $X^{\ast} (T) / M$ is torsion-free, and the assertion follows. 
\end{proof}

%We assume that the composite $T \hookrightarrow G \to G^{\ab}$ has a right inverse, that is, there exists a homomorphism $G^{\ab} \to T$ such that the composite $G^{\ab} \to T \to G^{\ab}$ is the identity. 
%Note that $G^{\ab} \cong Z^{\circ} / (Z^{\circ} \cap G^{\der})$, where $Z$ is the center of $G$, is a torus. 

We prove the first main theorem of this paper. 

\begin{theorem}
\label{qss_cryslift_fixeddet}
Let $G$ be a connected split reductive group over $\mathcal{O}_L$ and $T$ be a split maximal torus of $G$ over $\mathcal{O}_L$. 
%Assume that the composite $T \hookrightarrow G \twoheadrightarrow G^{\ab}$ has a right inverse. 

Let $\overline{\rho} \colon \Gal_K \to G(k_L)$ be a quasi-semisimple $G$-valued representation such that $\overline{\rho} (I_K) \subset T(k_L)$ and $\overline{\rho} (\Gal_K) \subset N_G (T)(k_L)$. Let $\overline{\rho}^{\ab} \colon \Gal_K \to G^{\ab} (k_L)$ be the abelianization of $\overline{\rho}$ and $\psi \colon \Gal_K \to G^{\ab} (\mathcal{O}_L)$ be a crystalline lift of $\overline{\rho}^{\ab}$. Then there exist a finite extension $L'/L$ and a crystalline lift $\rho \colon \Gal_K \to G(\mathcal{O}_{L'})$ with regular Hodge-Tate weights such that ${\rho}^{\ab} = \psi$. 
%of the ring of integers $\mathcal{O}_{L'}$ of $L'$. 
%There also exists a crystalline lift ${\rho}^{\reg}$ with regular Hodge-Tate weights such that ${\rho}^{\reg, \det} = \chi^{q^M} \psi^f$ for a crystalline representation $\chi \colon \Gal_K \to G^{\ab} (\mathcal{O}_L)$ and a sufficiently large positive integer $M$. 

Moreover, if we do not impose a crystalline lift $\rho$ has regular Hodge-Tate weights, $L'/L$ can be taken to be a finite unramified extension. 
\end{theorem}

\begin{proof}
\textbf{Step 1.\ }

First, we introduce notations in \cite[Proposition 2]{Lin20} needed in the proof of the theorem. 
%After replacing $L$ with its finite extension, we may assume that $T$ splits. 
%Write $W(G, T) \coloneqq W(G, T)(k_L) \subset W(G, T)(\overline{L})$. 
%We follows the notation and the proof of \cite[Proposition 2]{Lin20}. 
We naturally consider $M_{T, \crys}$ and $M_{T, k_L}$ as $\mathbb{Z}[W(G, T)] \otimes \mathbb{Z}[\Gal_K / I_K]$-modules. For an element $\tau \in W(G, T)$, write 
\begin{align*}
M_{T, \tau, \crys} &\coloneqq \ker(M_{T, \crys} \xrightarrow{\tau \otimes 1 - 1 \otimes \Phi_K} M_{T, \crys}), \\ 
M_{T, \tau, k_L} &\coloneqq \ker(M_{T, k_L} \xrightarrow{\tau \otimes 1 - 1 \otimes \Phi_K} M_{T, k_L})
\end{align*}
for subgroups of $M_{T, \crys}$ and $M_{T, k_L}$. 

%Since the representation $\overline{\rho}$ is quasi-semisimple, we may assume that $\overline{\rho}|_{I_K} \in M_{T, k_L}$, after replacing $L$ with its finite extension again.
% by Theorem \ref{reducible_quasi_semisimple} and Lemma \ref[(2)]{extendability_NGT}.
From here, let $\tau \in W(G, T)$ be the image of $\overline{\rho}(\Phi_K^{-1})$ by the quotient map $\zeta \colon N_G (T)(k_L) \to W(G, T)$. Then we have $\overline{\rho}|_{I_K} \in M_{T, \tau, k_L}$. 

Let $f \coloneqq [k_L : k]$ and $K_f$ be an unramified extension of $K$ of degree $f$. We may assume $\tau^f = 1$ after replacing $L$ with its finite unramified extension. We put 
\[
\Xi \coloneqq \sum_{i = 0}^{f - 1} \tau^i \otimes \Phi_K^{f - 1 - i}. 
\]
Let $M_{T, \crys}^0 \subset M_{T, \crys}$ be the $\mathbb{Z}[W(G, T)] \otimes \mathbb{Z}[\Gal_K / I_K]$-submodule which consists of continuous $T(\mathcal{O}_L)$-valued representations of $I_K$ that can be extended to a continuous representation $\Gal_{K_f} \to T(\mathcal{O}_L)$. We put $M_{T, \tau, \crys}^0 \coloneqq M_{T, \crys}^0 \cap M_{T, \tau, \crys}$. 
Note that the natural reduction map $M_{T, \crys}^0 \to M_{T, k_L}$ is surjective by Lemma \ref{fundamental_crys_lift}. 

Since $(\tau \otimes 1)^f = (1 \otimes \Phi_K)^f = \text{id}$ on both $M_{T, \crys}$ and $M_{T, k_L}$, we have an injection $\Xi M_{T, \crys}^0 \hookrightarrow M_{T, \tau, \crys}^0$. 
Then it follows that the composite
\[
\Xi M_{T, \crys}^0 \hookrightarrow M_{T, \tau, \crys}^0 \hookrightarrow M_{T, \tau, \crys} \to M_{T, \tau, k_L}
\]
is the same as the composite of the natural surjection $\Xi M_{T, \crys}^0 \to \Xi M_{T, k_L}$ and the natural isomorphism 
\[
\Xi M_{T, k_L} \cong M_{T, \tau, k_L}
\]
in the proof of \cite[Proposition 2]{Lin20}, and thus the composite $\Xi M_{T, \crys}^0 \to M_{T, \tau, k_L}$ is surjective. 

\vspace{3mm}

\leftline{\textbf{Step 2.\ }}
Next, we prove the existence of crystalline lifts with fixed abelianizations. 
%by Lin's argument. 
Fix isomorphisms $\mathbb{G}_m^{d_0} \times \mathbb{G}_m^{d_1} \cong G^{\ab} \times T' \cong T$ for a split torus $T'$ by Lemma \ref{right_inverse_existence}. 

For $A = k_L$ or $\mathcal{O}_L$ and a continuous representation $r \colon H \to T(A)$ of a locally profinite group $H$, we write $r_i \colon H \to \mathbb{G}_m (A) = A^{\ast}$ for the $i$-th component of $r$ via the isomorphism $\mathbb{G}_m^{d_0} \times \mathbb{G}_m^{d_1} \cong T$ for each $i \in [1, d_0 + d_1]$. 

Let $\overline{\rho}' \in M_{T, k_L}$ be a continuous representation such that $\Xi \overline{\rho}'$ is mapped to $\overline{\rho} |_{I_K}$ by the natural isomorphism $\Xi M_{T, k_L} \cong M_{T, \tau, k_L}$ in Step 1. 
%Set $\overline{\rho}' \coloneqq \Xi^{-1} \overline{\rho}$. 
We construct an element $\rho' \in M_{T, \crys}^0$ as follows:

\begin{itemize}
\item For $d_0 + 1 \leq i \leq d_0 + d_1$, take $\rho'_i \in M_{\mathbb{G}_m, \crys}^0$ to be mapped to $\overline{\rho}'_i$ by the surjection $M_{\mathbb{G}_m, \crys}^0 \to M_{\mathbb{G}_m, k_L}$ (see Lemma \ref{fundamental_crys_lift}). 

\item For $1 \leq i \leq d_0$, let $\psi_i \colon \Gal_K \to \mathcal{O}_L^{\ast}$ be the $i$-th component of $\psi$. The character $\psi'_i \colon K^{\ast} \to \mathcal{O}_L^{\ast}$ associated with $\psi_i$ may be written as the twist of an unramified character $\mathrm{unr}(c_i)$ with $x \mapsto \prod_{\sigma \colon K \to L} \sigma (\omega_K (x))^{a_{\sigma, i}}$, where $c_i \in \mathcal{O}_L^{\ast}$, $a_{\sigma, i} \in \mathbb{Z}$ for each $\sigma \in \Hom_{\mathbb{Q}_p} (K, L)$ by Lemma \ref{crystalline_character}. 
%Take an element $\widetilde{\sigma} \in \Hom_{\mathbb{Q}_p} (K_f, L)$ which extends $\sigma$ for each $\sigma \in \Hom_{\mathbb{Q}_p} (K, L)$. 
Then there exist distinct integers $b_{s, i}$ $(s \in \Hom_{\mathbb{Q}_p} (K_f, L))$ such that the character 
\[
\theta \colon \mathcal{O}_{K_f}^{\ast} \to \mathcal{O}_L^{\ast}, \quad x \mapsto \prod_{s \in \Hom_{\mathbb{Q}_p} (K_f, L)} s(x)^{b_{s, i}}
\]
satisfies the following conditions by \cite[Lemma 2.5]{BIP22}. 
\begin{itemize}
\item The integers $b_{s, i}$ $(s \in \Hom_{\mathbb{Q}_p} (K_f, L))$ are pairwise distinct. 
\item The character $\theta$ lifts the character $\mathcal{O}_{K_f}^{\ast} \to k_L^{\ast}$ associated with the composite of $I_{K_f} \xrightarrow{=} I_K$ and $\overline{\rho}'_i$. 
\item $\theta(x) = \psi'_i (x)$ for all $x \in \mathcal{O}_K^{\ast}$. 
\end{itemize}

Let $\rho'_i \colon I_K \to \mathcal{O}_L^{\ast}$ 
%be the extension to $\Gal_{K_f}$ (which exists by Lemma \ref{G_extendability}) of 
be the composite of the following maps
\[
%I_K = I_{K_f} \hookrightarrow W_{K_f} \twoheadrightarrow W_{K_f}^{\ab} \xrightarrow[\cong]{\rec_{K_f}^{-1}} K_f^{\ast} \xrightarrow{\prod_{\tau \colon K_f \to L} \tau(-)^{b_{\tau, i}}} L^{\ast},  
I_K = I_{K_f} \hookrightarrow W_{K_f} \twoheadrightarrow W_{K_f}^{\ab} \xrightarrow[\cong]{\rec_{K_f}^{-1}} K_f^{\ast} \xrightarrow{\prod_{s \colon K_f \to L} s(\omega_{K_f} (-))^{b_{s, i}}} L^{\ast},  
\]
%where $\tau \in \Hom_{\mathbb{Q}_p} (K_f, L)$ and $b_{\tau, i} = a_{\sigma, i}$ if $\tau = \widetilde{\sigma}$ for some $\sigma \in \Hom_{\mathbb{Q}_p} (K, L)$, otherwise $b_{\tau, i} = 0$. 
%We set $\rho'_i \coloneqq \widetilde{\rho}'_i |_{I_K}$. 
%it can be constructed by twisting powers of the Galois conjugates of the Lubin-Tate characters, which are the composite of the unramified character $\text{unr}(c)$ and 
\end{itemize}

Then the continuous $T(\mathcal{O}_L)$-valued representation $\rho' \in M_{T, \crys}^0$ constructed as above satisfies the following properties:
\begin{itemize}
\item The image of $\rho'$ by the surjection $M_{T, \crys}^0 \to M_{T, k_L}$
%by the map $M_{T, \crys} \to M_{T, k_L}$ 
is $\overline{\rho}'$. 
%and $\rho' \in M_{T, \crys}^0$ such that the image of $\rho'$ by the map $M_{T, \crys} \to M_{T, k_L}$ is $\overline{\rho}'$
\item The abelianization ${\rho'}^{\ab} \colon I_K \to G^{\ab} (\mathcal{O}_L)$ is extendable to a crystalline representation $\psi' \colon \Gal_{K_f} \to G^{\ab} (\mathcal{O}_L)$ such that $\psi |_{W_K}$ is the same as the composite 
\[
W_K \twoheadrightarrow W_K^{\ab} \xrightarrow[\cong]{\rec_K^{-1}} K^{\ast} \hookrightarrow K_f^{\ast} \xrightarrow[\cong]{\rec_{K_f}} W_{K_f}^{\ab} \xrightarrow{\psi' |_{W_{K_f}}} G^{\ab} (\mathcal{O}_L).
\]
\end{itemize} 

%In fact, the continuous character $K^{\ast} \to \mathcal{O}_L^{\ast}$ associated with a crystalline character $\Gal_K \to \mathcal{O}_L^{\ast}$ can be written by the twist of an unramified character $\text{unr}(c)$, which sends the uniformizer $\varpi_K$ to an element $c \in \mathcal{O}_L^{\ast}$, with $x \mapsto \prod_{\sigma \colon K \to L} \sigma (x)^{a_{\sigma}}$ where $a_{\sigma} \in \mathbb{Z}$ for each $\sigma \in \Hom_{\mathbb{Q}_p} (K, L)$. 

Let $R'$ and $\Xi R'$ be the $T(\mathcal{O}_L)$-valued crystalline representations $\Gal_{K_f} \to T(\mathcal{O}_L)$ extended from $\rho'$ and $\Xi \rho'$ respectively as above. By the calculation in the proof of \cite[Lemma 2.1]{BIP22}, we have 
\begin{align*}
(\Xi R')^{\ab} (\rec_{K_f} (y)) 
&= \prod_{\sigma \in \Gal(K_f / K)} {R'}^{\ab} (\sigma \rec_{K_f} (y) \sigma^{-1}) \\ 
&= \prod_{\sigma \in \Gal(K_f / K)} {R'}^{\ab} (\rec_{K_f} (\sigma(y))) \\ 
&= {R'}^{\ab}(\rec_{K_f} (\text{Nm}_{K_f / K} (y))) \\ 
&= \psi |_{W_K} (\rec_K (\text{Nm}_{K_f / K} (y)))) \\ 
&= \psi |_{W_{K_f}} (\rec_{K_f} (y))
\end{align*}
for any $y \in K_f^{\ast}$. Here, the first equality follows from the fact that each $\tau^i$ acts trivially on the abelianization ${\rho'}^{\ab}$, the fourth equality follows from the second bullet of the properties of the representation $\rho' \in M_{T, \crys}^0$, and the last equality follows from $\text{Nm}_{K_f / K} \circ \rec_{K_f}^{-1} = \rec_K^{-1} \circ \lambda$ by local class field theory, where $\lambda \colon W_{K_f}^{\ab} \to W_K^{\ab}$ is induced by the inclusion $W_{K_f} \hookrightarrow W_K$. So we have $(\Xi R')^{\ab} |_{W_{K_f}} = \psi |_{W_{K_f}}$, and by continuity $(\Xi R')^{\ab} = \psi |_{\Gal_{K_f}}$. 
%by the proof of \cite[Lemma 2.1]{BIP22}. 
Moreover, it follows that the image of $\Xi \rho'$ by the surjection $\Xi M_{T, \crys}^0 \to M_{T, \tau, k_L}$ is $\Xi \overline{\rho}' = \overline{\rho} |_{I_K}$ by the construction. 

Let $w \in N_G (T)(\mathcal{O}_L)$ be an element which is a lift of the element $\overline{\rho}(\Phi_K)^{-1}$. 
We may assume that the image of $w$ by the composite $N_G (T)(\mathcal{O}_L) \hookrightarrow G(\mathcal{O}_L) \to G^{\ab} (\mathcal{O}_L)$ is $\psi (\Phi_K)^{-1}$ by the assumption of the existence of a right inverse of the map $T \to G^{\ab}$. 

Since $\Xi M_{T, \crys}^0 \subset M_{T, \tau, \crys}^0$, there exists a continuous representation $\rho \colon \Gal_K \to N_G (T)(\mathcal{O}_L) \subset G(\mathcal{O}_L)$ such that $\rho |_{\Gal_{K_f}} = \Xi R'$ and $\rho^{\ab} = \psi$ by Lemma \ref{extendability_NGT} (1) where $w$ is taken as above. 
Then $\rho$ is a crystalline lift of $\overline{\rho}$ satisfying $\rho^{\ab} = \psi$. 
But the Hodge-Tate weights of $\rho$ are not necessarily regular. 

\vspace{3mm} 

\leftline{\textbf{Step 3.\ }}

Finally, we show the existence of a crystalline lift with fixed abelianization which has regular Hodge-Tate weights. We prove it by an argument of Lin \cite[Lemma 12, Theorem 6]{Lin20}. 

Let $i \colon T \hookrightarrow G$ be the embedding and fix embeddings $K_f \subset L \subset \mathbb{C}_p$. 
%fix embeddings $\iota \colon K_f \hookrightarrow L$ and $\sigma_0 \colon L \hookrightarrow \mathbb{C}_p$. 
For constructing such a crystalline lift with regular Hodge-Tate weights, we first show the existence of a crystalline representation $t \colon \Gal_{K_f} \to T(\mathcal{O}_L)$ satisfying the following properties for each $\sigma \in \Hom_{\mathbb{Q}_p} (L, \mathbb{C}_p)$:
\begin{itemize}
\item $\mathcal{H}\mathcal{T}(i(t))^{\sigma}$ is regular if $\sigma |_{K_f}$ coincides with the fixed embedding $K_f \hookrightarrow \mathbb{C}_p$. 
\item $\mathcal{H}\mathcal{T}(i(t))^{\sigma}$ is the trivial cocharacter if $\sigma |_{K_f}$ does not coincide with the fixed embedding $K_f \hookrightarrow \mathbb{C}_p$. 
\item $t^{\ab}$ is trivial. 
\end{itemize}

The existence of such $t$ is reduced to show $\ker(X_{\ast}^{\ab}) \not\subset \bigcup_{w \in W(G, T) \backslash \{ e \}} \ker(w-1)$
%where $X_{\ast}^{\ab} \colon X_{\ast} (T) \to X_{\ast} (G^{\ab})$ is the map induced by the composite $T \hookrightarrow G \twoheadrightarrow G^{\ab}$ and 
%\[
%w-1 \colon X_{\ast} (T) \to X_{\ast} (T), \ \chi \mapsto \ad(w)(\chi) \cdot \chi^{-1}
%\]
%for each $w \in W(G, T)$.  
by considering powers of the fundamental character of $K_f$ for each coordinate of $T(\mathcal{O}_L) \cong (\mathcal{O}_L^{\ast})^{d_0 + d_1}$ (see the proof of \cite[Lemma 12]{Lin20}). 
In fact, let 
\[
x \in \ker(X_{\ast}^{\ab}) \backslash \left( \bigcup_{w \in W(G, T) \backslash \{ e \} } \ker(w-1) \right).
\]
Then the composite 
\[
W_{K_f} \twoheadrightarrow W_{K_f}^{\ab} \xrightarrow[\cong]{\rec_{K_f}^{-1}} K_f^{\ast} \hookrightarrow L^{\ast} \cong \mathbb{G}_m (L) \xrightarrow{x(L)} T(L)
\]
extends to a required crystalline representation $t \colon \Gal_{K_f} \to T(L)$. 
The assertion follows from Lemma \ref{trivdet_regular}. 

For such $t$, let $v_0 \in M_{T, \crys}^0$ be the composite of the following maps 
\[
I_K = I_{K_f} \xrightarrow{t |_{I_{K_f}}} T(\mathcal{O}_L). 
\]
Then we have $v_0^{\ab} = 1$. For a sufficiently large positive integer $M$, the representation $v' \coloneqq v \cdot (\Xi v_0)^{MN}$, where $v \coloneqq \rho |_{I_K}$ and $N$ is the cardinality of $k_L^{\ast}$, satisfies ${v'}^{\ab} = (v \cdot (\Xi v_0)^{MN})^{\ab} = v^{\ab}$ and $\overline{v'} = \overline{v}$. Since $v, \Xi v_0 \in \Xi M_{T, \crys}^0 \subset M_{T, \tau, \crys}^0$, the representation $v'$ extends to a continuous representation $\rho'' \colon \Gal_K \to G(\mathcal{O}_L)$, which is a crystalline lift of $\overline{\rho}$ by $K_f \subset L$ and \cite[Corollary 2]{Lin20}. 
%We have ${\rho''}^{\ab} |_{I_K} = v^{\ab} = \rho^{\ab} |_{I_K}$ by the construction. 
We have ${\rho''}^{\ab} |_{I_K} = \rho^{\ab}|_{I_K}$ by the construction. 
So replacing $L$ with its finite extension, there is an unramified representation $u \colon \Gal_K \to Z^{\circ} (\mathcal{O}_L)$ whose composite with the map $Z^{\circ} (\mathcal{O}_L) \to G^{\ab} (\mathcal{O}_L)$ is $(\rho''^{\ab})^{-1} \cdot \rho^{\ab}$ since $Z^{\circ} \to G^{\ab}$ is an isogeny of algebraic tori. 
Then $\rho_{\reg} \coloneqq u \rho''$ is a crystalline lift of $\overline{\rho}$ which satisfies $\rho_{\reg}^{\ab} = \rho^{\ab} = \psi$ by the construction.  
In addition, the representation $\rho_{\reg}$ has regular Hodge-Tate weights since for each $\sigma \in \Hom_{\mathbb{Q}_p} (L, \mathbb{C}_p)$, there exists an integer $i \in [0, f-1]$ and an element $\sigma' \in \Hom_{K_f} (L, \mathbb{C}_p)$ such that $\sigma = \sigma' \circ \Phi_K^{f-1-i}$ and 
\[
\mathcal{H}\mathcal{T}(\Xi v_0)^{\sigma} = \tau^i \mathcal{H}\mathcal{T}(v_0)^{\sigma'} \tau^{-i}
\]
by the calculation in the proof of \cite[Theorem 6]{Lin20}. 
\end{proof}

\section{Potentially crystalline lifts of semisimple $L$-parameters with fixed abelianizations}

In this section, we also prove the existence of potentially crystalline lifts with a fixed potentially crystalline abelianization of semisimple mod $p$ $L$-parameters for connected quasi-split tame groups. 
%, which splits over a tame Galois extension $F/K$. 

First, we recall the definitions of $L$-groups and $L$-parameters following \cite[\S 2, \S 3]{Lin23b}. 
Let $K/\mathbb{Q}_p$ be a finite extension and $G$ be a connected quasi-split reductive group over $K$. Fix a $K$-pinning $(B, T, \{ X_{\alpha} \})$ of $G$. 
The pinned group $(G, B, T, \{ X_{\alpha} \} )$ has a dual pinned group $(\widehat{G}, \widehat{B}, \widehat{T}, \{ Y_{\alpha} \} )$ defined over $\mathbb{Z}$. 
Let $E \subset \overline{K}$ be a splitting field of $G$, which is a finite Galois extension over $K$. 
The Galois action on the based root datum induces a Galois action $\Gal(E/K) \to \Aut(\widehat{G}, \widehat{B}, \widehat{T}, \{ Y_{\alpha} \} )$. 
Write ${^L}G \coloneqq \widehat{G} \rtimes \Gal_K$ (see \cite{BG10} and \cite[\S 1.1]{Zhu20a} for more details). Write ${^L} \{ \ast \} \cong \Gal_K$ for the $L$-group of the trivial group. 

\begin{definition}[cf.\ {\cite[Definition 3.1.1, Definition 3.1.3]{Lin23b}}]
Let $G$, $H$ be connected reductive groups over $K$ and $\widehat{G}$, $\widehat{H}$ be their Langlands dual groups. 
\begin{enumerate}
\item For a finite ring $A$, a group homomorphism 
\[
\widehat{G}(A) \rtimes \Gal(E/K) \to \widehat{H}(A) \rtimes \Gal(E/K)
\]
is said to be an \emph{$L$-homomorphism (with $A$-coefficients)} if each element $g \times \sigma \in \widehat{G}(A) \rtimes \Gal(E/K)$ is sent to $h \times \sigma \in \widehat{H}(A) \rtimes \Gal(E/K)$. 
\item For a finite ring $A$, a group homomorphism 
\[
\widehat{G}(A) \rtimes W_{E/K} \to \widehat{H}(A) \rtimes W_{E/K}
\]
is said to be an \emph{$L$-homomorphism (with $A$-coefficients)} if each element $g \times \sigma \in \widehat{G}(A) \rtimes \Gal(E/K)$ is sent to $h \times \sigma \in \widehat{H}(A) \rtimes \Gal(E/K)$. 
\item For a finite ring $A$, a group homomorphism 
\[
\widehat{G}(A) \rtimes \Gal_K = {^L}G(A) \to {^L}H(A) = \widehat{H}(A) \rtimes \Gal_K
\]
is said to be an \emph{$L$-homomorphism (with $A$-coefficients)} if each element $g \times \sigma \in \widehat{G}(A) \rtimes \Gal_K$ is sent to $h \times \sigma \in \widehat{H}(A) \rtimes \Gal_K$, and there exists an open subgroup $\Gamma \subset \Gal_K$ such that $1 \times \sigma \in \widehat{G}(A) \rtimes \Gal_K$ is sent to $1 \times \sigma \in \widehat{H}(A) \rtimes \Gal_K$ for every $\sigma \in \Gamma$. 
\item An \emph{$L$-parameter (of $\Gal_K$ (resp.\ $\Gal(E/K)$, $W_{E/K}$) for $G$ with $A$-coefficients)} is an $L$-homomorphism 
\begin{align*}
&\Gal_K \simeq {^L} \{ \ast \} (A) \to {^L}G(A), \\ 
\text{(resp.\ } \Gal(E/K) &\to \widehat{G}(A) \rtimes \Gal(E/K), \ W_{E/K} \to \widehat{G}(A) \rtimes W_{E/K}), 
\end{align*}
defined up to $\widehat{G}(A)$-conjugacy. Note that an $L$-parameter can be considered as an ${^L}G(A)$-valued Galois representation $\Gal_K \to {^L}G(A)$. 
\item For a profinite ring $R$, an \emph{$L$-parameter (for $G$ with $R$-coefficients)} is a compatible system of $L$-parameters $\Gal_K \to {^L}G(A)$, where $A$ is a finite quotient of $R$. 
\item For $A = \overline{\mathbb{F}}_p$ or $\overline{\mathbb{Z}}_p$, an \emph{$L$-parameter (for $G$ with $A$-coefficients)} is the composite of an $L$-parameter $\Gal_K \to {^L}G(A')$ for a finitely generated $\mathbb{Z}_p$-algebra $A' \subset A$ and the inclusion ${^L}G(A') \hookrightarrow {^L}G(A)$. 
\end{enumerate}

\end{definition}

%In \cite[\S 2.1.1]{Lin23b}, for an algebraic group $H$ over a commutative ring $R$ and a morphism of schemes $f \colon \mathbb{G}_m \to H$, the functors $P_H (f)$, $U_H (f)$, $Z_H (f)$ on the $R$-algebras, which are subfunctor of $H$ are constructed. 

%\begin{definition}[{\cite[Definition 2.2.1]{Lin23b}}]
%Let $H$ be a linear algebraic group over a commutative ring $R$. 
%A subgroup of $H$ is said to be a \emph{pseudo-parabolic} (resp.\ a \emph{pseudo-Levi}) subgroup if it is of the form $P_H (f)$ (resp.\ $Z_H (f)$) for some cocharacter $f \colon \mathbb{G}_m \to H^{\circ}$. 
%\end{definition}

%\begin{definition}[{\cite[Definition 2.2.3]{Lin23b}}]
%Let $H$ be a linear algebraic group over an algebraically closed field. 
%\begin{enumerate}
%\item A pseudo-parabolic $P$ of $H$ is said to be a \emph{proper} pseudo-parabolic if $P \cap H^{\circ} \neq H^{\circ}$. 
%\item An abstract subgroup $\Gamma \subset H$ is said to be \emph{pseudo-completely reducible} in $H$ if whenever $\Gamma$ is contained in a proper pseudo-parabolic, it is also contained in a corresponding pseudo-Levi.  
%\end{enumerate}
%\end{definition}

%\begin{definition}[{\cite[Definition 3.2.5]{Lin23b}}]
%An $L$-parameter $\overline{\rho} \colon \Gal_K \to {^L}G(\overline{\mathbb{F}}_p)$ is said to be \emph{semisimple} if the image of $\overline{\rho}$ is pseudo-completely irreducible. 
%\end{definition}

\begin{definition}[cf.\ {\cite[Definition 5.1.1]{Lin23b}}]
Let $\Lambda \subset \overline{\mathbb{Z}}_p$ be a discrete valuation ring. 
For a connected quasi-split reductive group $G$ over $K$ which splits over a finite Galois extension $E/K$, an $L$-parameter $\rho \colon \Gal_K \to \widehat{G} (\Lambda) \rtimes \Gal(E/K)$ is said to be \emph{crystalline} (resp.\ \emph{potentially crystalline}) if for some closed embeddings of algebraic groups $\widehat{G} \rtimes \Gal_{E/K} \hookrightarrow \GL_d$ for some $d \geq 1$, the composite $\Gal_K \to \widehat{G} (\Lambda) \rtimes \Gal(E/K) \to \GL_d (\Lambda)$ is crystalline (resp.\ \emph{potentially crystalline}). 
\end{definition}
%alg. grp.の埋め込みは\Lambaのfractional field上で考えれば十分そうだが、Z上でもよかったか?

An $L$-parameter $\overline{\rho} \colon \Gal_K \to {^L}G(\overline{\mathbb{F}}_p)$ is said to be \emph{semisimple} if $\overline{\rho}$ is semisimple in the sense of \cite[Definition 3.2.5]{Lin23b}. 

The following Langlands-Shelstad factorizations of $L$-parameters are shown in \cite{Lin23b}. 

\begin{theorem}[{\cite[Theorem 3.4.1]{Lin23b}}]
\label{LS_factorization}
For a semisimple $L$-parameter $\overline{\rho} \colon \Gal_K \to {^L}G(\overline{\mathbb{F}}_p)$, there exists a tame $K$-torus $S$ of $G$ and a canonical $L$-homomorphism ${^L}j \colon {^L}S_{\overline{\mathbb{F}}_p} \to {^L}G_{\overline{\mathbb{F}}_p}$ such that $\overline{\rho}$ factors through ${^L}j$. 
\end{theorem}

For a semisimple $L$-parameter $\overline{\rho} \colon \Gal_K \to {^L}G(\overline{\mathbb{F}}_p)$, we fix a tame $K$-torus $S \subset G$ satisfying the condition of Theorem \ref{LS_factorization} in the assertion of Theorem \ref{qst_dRlift_fixeddet}. 

For $A = \overline{\mathbb{Z}}_p^{\ast}$ or $\overline{\mathbb{F}}_p^{\ast}$, note that the module $H_{\cont}^1 (W_{E/K}, X^{\ast} (T_E) \otimes_{\mathbb{Z}} A)$ (resp.\ $H_{\cont}^1 (\Gal_K, X^{\ast} (T_E) \otimes_{\mathbb{Z}} A)$) is naturally bijective to the set of $\widehat{G}(A)$-conjugacy classes of $L$-parameters $W_{E/K} \to \widehat{G}(A) \rtimes W_{E/K}$ (resp.\ $\Gal_K \to {^L}G(A) \rtimes \Gal_K$). 
We also use the $p$-adic local Langlands correspondence for tori with divisible coefficients. 

\begin{theorem}[{\cite{Bir20}}]
%[cf.\ {\cite{Bir20}, \cite[Theorem 4.3.1]{Lin23b}}]
\label{padicLLC_tori_divcoef}
For a $K$-torus $T$ with splitting field $E$, which is a finite Galois extension over $K$, and a divisible abelian topological group $D$, there exists a natural isomorphism 
\[
H_{\cont}^1 (W_{E/K}, X^{\ast} (T_E) \otimes_{\mathbb{Z}} D) \xrightarrow{\sim} \Hom_{\cont} (T(K), D). 
\]
Write $T_E \coloneqq T \times_K E$. Then the following natural maps 
\begin{align*}
H_{\cont}^1 (\Gal_K, X^{\ast} (T_E) \otimes_{\mathbb{Z}} \overline{\mathbb{F}}_p^{\ast}) &\to H_{\cont}^1 (W_{E/K}, X^{\ast} (T_E) \otimes_{\mathbb{Z}} \overline{\mathbb{F}}_p^{\ast}), \\ 
H_{\cont}^1 (\Gal_K, X^{\ast} (T_E) \otimes_{\mathbb{Z}} \overline{\mathbb{Z}}_p^{\ast}) &\to H_{\cont}^1 (W_{E/K}, X^{\ast} (T_E) \otimes_{\mathbb{Z}} \overline{\mathbb{Z}}_p^{\ast}), 
\end{align*}
are also isomorphisms. 
\end{theorem}

\begin{proof}
For the first isomorphism, see \cite{Bir20}, \cite[Theorem 4.3.1]{Lin23b}. 
%Let $k'/\mathbb{F}_p$ be a finite extension 
Let $A$ be a finite discrete ring.  
%and $\overline{\rho} \colon W_{E/K} \to \widehat{T} (A) \rtimes W_{E/K}$ be an $L$-homomorphism. 
Then the set of $L$-homomorphisms $\rho \colon W_{E/K} \to \widehat{T} (A) \rtimes W_{E/K}$ is naturally bijective to the set of $L$-homomorphisms $\rho \colon W_K \to \widehat{T} (A) \rtimes W_K$ since $W_E$ acts trivially on $\widehat{T}$. 
%let $E'/E$ be a finite extension such that $\overline{W_E^{\der}} \subset W_{E'} \subset W_E$ and each $\sigma \in W_{E'}$ is sent to the element $1 \times \sigma \in \widehat{T} (\overline{\mathbb{F}}_p^{\ast}) \rtimes W_{E/K}$. 
%Note that there is the natural surjection $W_{E/K} \twoheadrightarrow W_K / W_{E'} \cong \Gal(E'/K)$. 
%For $A = \overline{\mathbb{F}}_p^{\ast}$ or $\overline{\mathbb{Z}}_p^{\ast}$, 
%hen the $L$-homomorphism $\overline{\rho} \colon W_{E/K} \to \widehat{T} (A) \rtimes W_{E/K}$ naturally induce the map $\Gal(E'/K) \to \widehat{T} (A) \rtimes \Gal(E'/K)$ 
%by the surjection $W_{E/K} \twoheadrightarrow \Gal(E/K)$ 
%by the construction. 
%since $\Gal_E$ acts trivially on $\widehat{G}$. 

By \cite[Lemma 3.1.2 (2), (4)]{Lin23b}, an $L$-homomorphism $\rho \colon W_K \to \widehat{T}(A) \rtimes W_K$ naturally corresponds to an $L$-parameter $\rho \colon \Gal_K \to {^L}G(A)$. So we have an isomorphism 
\begin{equation}
\label{padicLLC_coef_A}
H_{\cont}^1 (\Gal_K, X^{\ast} (T_E) \otimes_{\mathbb{Z}} A) \xrightarrow{\sim} H_{\cont}^1 (W_{E/K}, X^{\ast} (T_E) \otimes_{\mathbb{Z}} A)
\end{equation}
It follows that the last maps are isomorphisms by considering the (co)limits of the isomorphisms (\ref{padicLLC_coef_A}) for finite discrete coefficients $A$. 
%It follows that the second map of the theorem it suffices to show the following map 
%\begin{equation}
%\label{A_coef_L_param_map}
%H_{\cont}^1 (\Gal_K, X^{\ast} (T) \otimes_{\mathbb{Z}} A) \to H_{\cont}^1 (W_{E/K}, X^{\ast} (T) \otimes_{\mathbb{Z}} A)
%\end{equation}
%is an isomorphism. This follows because each $L$-parameter $\Gal(E/K) \to \widehat{T}(A) \rtimes \Gal(E/K)$ is uniquely lifted to an $L$-parameter $\Gal_K \to {^L}T(A)$ by \cite[Lemma 3.1.2 (4)]{Lin23b}. 
%It follows that the last map of the theorem is an isomorphism since the map \ref{A_coef_L_param_map} is an isomorphism and each $L$-parameter $\rho \colon $
\end{proof}

For a $K$-torus $S$ and an $L$-parameter $\overline{\rho} \colon \Gal_K \to {^L}S (\overline{\mathbb{F}}_p)$, let $[\overline{\rho}] \colon \Gal_K \to {^L}S (\overline{\mathbb{Z}}_p)$ be the $L$-parameter induced by the isomorphisms in Theorem \ref{padicLLC_tori_divcoef} for $D = \overline{\mathbb{F}}_p^{\ast}$ and the Teichm\"uller character $\overline{\mathbb{F}}_p^{\ast} \hookrightarrow \overline{\mathbb{Z}}_p^{\ast} (\subset W(\overline{\mathbb{F}}_p)^{\ast})$ from $\overline{\rho}$. 

\begin{theorem}
\label{qst_dRlift_fixeddet}
Let $G$ be a connected quasi-split tame group over $K$ and $\overline{\rho} \colon \Gal_K \to {^L}G (\overline{\mathbb{F}}_p)$ be a semisimple mod $p$ $L$-parameter. 
%Let $T$ be a maximal torus of $G$. 
%Assume that the composite $T \hookrightarrow G \twoheadrightarrow G^{\ab}$ has a right inverse. 
Fix a tame $K$-torus $S$ of $G$ such that $\overline{\rho}$ factors through ${^L}S (\overline{\mathbb{F}}_p)$.
%, which exists by Theorem \ref{LS_factorization}. 
Let $i \colon G^{\ab} \to S$ be a right inverse of the composite $S \hookrightarrow G \twoheadrightarrow G^{\ab}$. 
%Assume that the composite $S \hookrightarrow G \twoheadrightarrow G^{\ab}$ has a right inverse, which is written as $i \colon G^{\ab} \to S$. 

For $A = \overline{\mathbb{F}}_p^{\ast}$ or $\overline{\mathbb{Z}}_p^{\ast}$ and an $L$-parameter $\tau \colon \Gal_K \to {^L}S(A)$, write $\tau^{\ab}$ for the composite of $\tau$ and the map ${^L}S (A) \to {^L}G^{\ab}(A)$ induced by the composite $G^{\ab} \xrightarrow{i} S \hookrightarrow G$ and we set $\overline{\psi} \coloneqq \overline{\rho}^{\ab}$. 

Let $\psi \colon \Gal_K \to {^L}G^{\ab}(\overline{\mathbb{Z}}_p)$ be a potentially crystalline lift of $\overline{\psi}$. 
Then $\overline{\rho}$ admits a potentially crystalline lift $\rho \colon \Gal_K \to {^L}G(\overline{\mathbb{Z}}_p)$ which has regular Hodge-Tate weights and satisfies $\rho^{\ab} = \psi$. 
\end{theorem}

\begin{proof}
%By Theorem \ref{LS_factorization}, $\overline{\rho}$ factors through ${^L}S(\overline{\mathbb{F}}_p)$ for some $K$-torus $S$ of $G$. 
Let $E$ be a finite Galois extension of $K$ such that the image of $\overline{\rho}$ is contained in ${^L}G(E)$, $S$ splits over $E$, $\psi|_{\Gal_E}$ is crystalline, the character $\Gal_K \to k^{\ast}$ associated with the mod $\varpi_K$ fundamental character is trivial on $\Gal_E$. 
We fix an embedding $E \hookrightarrow \overline{K}$ and a uniformizer $\varpi_E \in E$. 

For a $K$-torus $T$ which splits over $E$, let $\text{Nm}_{E/K} \colon T(E) \to T(K)$ be the norm map induced by $\text{Nm}_{E/K} \colon \text{Res}_{E/K} T_E \to T$, where $T_E \coloneqq S \times_K E$.
By Theorem \ref{padicLLC_tori_divcoef}, we have the following commutative diagram 
\[
\xymatrix{
H_K^1 (G^{\ab}) \ar[r]^-{\cong} \ar[d]^{\text{-}|_{\Gal_E}} & 
\Hom_{\cont} (G^{\ab} (K), \overline{\mathbb{Z}}_p^{\ast}) \ar[r] \ar[d]^{\text{-} \circ \text{Nm}_{E/K}} & 
\Hom_{\cont} (S(K), \overline{\mathbb{Z}}_p^{\ast}) \ar[d]^{\text{-} \circ \text{Nm}_{E/K}} & 
H_K^1 (S) \ar[d]^{\text{-}|_{\Gal_E}} \ar[l]_-{\cong} \\ 
H_E^1 (G^{\ab}) \ar[r]^-{\cong} & 
\Hom_{\cont} (G^{\ab} (E), \overline{\mathbb{Z}}_p^{\ast}) \ar[r] & 
\Hom_{\cont} (S(E), \overline{\mathbb{Z}}_p^{\ast}) & 
H_E^1 (S) \ar[l]_-{\cong}
}
\]
where $H_K^1 (-) \coloneqq H_{\cont}^1 (\Gal_K, X^{\ast} (- \times_K \overline{K}) \otimes_{\mathbb{Z}} \overline{\mathbb{Z}}_p^{\ast})$.  
%Since the Langlands dual $\widehat{(-)}$ is functorial and exact on the category of $K$-tori, 

By the existence of the right inverse of the map $S \to G^{\ab}$, we have $S \cong G^{\ab} \times S'$ for some $K$-torus $S'$. 
%Taking a tensor product of powers of the $\Sigma$-conjugates of the Lubin-Tate characters, we have a crystalline lift $\rho'_E \in H_E^1 (S)$ such that ${\rho'}_E^{\ab} = \psi |_{\Gal_E}$. we may assume that the $H_E^1 (S')$-component of $\rho'_E$ is in the image of the map induced by $\text{Nm} |_{E/K}$ and let $\rho' \in H_K^1 (S)$ be an inverse image of $\rho'_E$. 
Let $\overline{\rho}_{S'} \colon \Gal_K \to {^L}S' (\overline{\mathbb{F}}_p)$ be the ${^L}S' (\overline{\mathbb{F}}_p)$-component of $\overline{\rho}$.
Then $[\overline{\rho}_{S'}] \colon \Gal_K \to {^L}S' (\overline{\mathbb{Z}}_p)$ is a potentially crystalline $L$-parameter with its reduction $\overline{\rho}_{S'}$ since its image is finite. 
%Take $\rho'_1 \in H_K^1 (S')$ to be crystalline with its reduction $\overline{\rho}_{S'}$. 
%Taking a potentially crystalline $\rho'_{S'} = \rho'_1 \rho'_2 \in H_K^1 (S')$ where $\rho'_1$ is a crystalline with its reduction $\overline{\rho}'_1$ and $\rho'_2 \coloneqq \overline{\rho}_{S'}[\overline{\rho}'_1]^{-1}$ where $[-] \colon \overline{\mathbb{F}}_p^{\ast} \to W(\overline{\mathbb{F}}_p^{\ast}) \hookrightarrow \overline{\mathbb{Z}}_p$ is the Teichm\"uller character and 
Let 
\[
\rho' \coloneqq (\psi, [\overline{\rho}_{S'}]) \in H_K^1 (G^{\ab}) \oplus H_K^1 (S') \cong H_K^1 (S). 
\]
It follows that $\rho'$ is a potentially crystalline lift of $\overline{\rho}$ and we have ${\rho'}^{\ab} = \psi$ by the construction. 

Note that $\text{Gal} (E/K)$ acts on $X_{\ast} (\widehat{S}_E) = X^{\ast} (S_E)$. 
For a character $x \in X^{\ast} (S_E)$, let $\widetilde{\chi}_x \colon S(E) \to \overline{\mathbb{Z}}_p^{\ast}$ be the composite of a character $x(E) \colon S(E) \to E^{\ast}$, the fundamental character $E^{\ast} \to \mathcal{O}_E^{\ast}$, and a fixed inclusion $\mathcal{O}_E^{\ast} \hookrightarrow \overline{\mathbb{Z}}_p^{\ast}$. 

%It suffices to show the existence of a cocharacter $x \in X_{\ast} (\widehat{S}_E)$ such that for any $\sigma \in \text{Gal} (E/K)$, the cocharacter $\sigma \cdot x \in X_{\ast} (\widehat{S}_E)$ has regular Hodge-Tate weights and the trivial abelianization since $\rho' {\rho''}^M$, where $\rho'' \in H_K^1 (S)$ is the corresponding potentially crystalline representation to $\chi_x \coloneqq \widetilde{\chi}_x |_{S(K)}$, has a regular Hodge-Tate weights for some integer $M$ by \cite[Lemma 5.2.6]{Lin23b}, since the $L$-parameter $\rho' \chi_x [\chi_x]^{-1}$ for such $x \in X^{\ast} (S_E)$ is what we want. 

By Step 3 of the proof of Theorem \ref{qss_cryslift_fixeddet}, there exists a cocharacter $x \in X_{\ast} (\widehat{S}_E)$ which is regular, i.e., not killed by any roots of $\widehat{S}_E$, and the trivial abelianization. 

It follows that for any $\sigma \in \text{Gal} (E/K)$, the cocharacter $\sigma \cdot x \in X_{\ast} (\widehat{S}_E)$ is also regular, and the $L$-parameter corresponding to $\chi_x \coloneqq \widetilde{\chi}_x |_{S(K)}$ has regular Hodge-Tate weights by \cite[Lemma 5.2.6]{Lin23b}. 

Then $\rho' {\rho''}^M$, where $\rho'' \in H_K^1 (S)$ is the potentially crystalline representation corresponding to $\chi_x$, has regular Hodge-Tate weights for a sufficiently large integer $M$. 
%For such $x \in X^{\ast} (S_E)$, let $\tau \colon \Gal_K \to {^L}S(\overline{\mathbb{Z}}_p)$ be the $L$-parameter associated with the character $\chi_x [\chi_x]^{-1} \colon S(K) \to \overline{\mathbb{Z}}_p^{\ast}$. 
The $L$-parameter $\rho' {\rho''}^M [\overline{\rho}'']^{-M}$ is what we want.
%Since $\text{Nm}_{E/K}$ is the $[E : K]$-th power map on the $K$-rational points, the theorem follows by taking a crystalline lift of $\overline{\rho}_{\mathcal{E}}$ such that its Hodge-Tate weights are regular and its determinant is $\psi$ and its projection to $\Hom_{\cont} (S'(E), \overline{\mathbb{Z}}_p)$ is the $[E : K]$-th power of a crystalline representation, as in Theorem \ref{qss_cryslift_fixeddet}. 
\end{proof}

%ellipticyのdef
%thm 3.10をfixed det.付きかつ一般のpの場合に書き直し

\end{document}